\documentclass{amsart}
\usepackage{mathrsfs}
\usepackage{amssymb}
\theoremstyle{plain}
\newtheorem{prop}{Proposition}[section]
\newtheorem*{prop*}{Proposition}
\newtheorem{thm}[prop]{Theorem}
\newtheorem*{thm*}{Theorem}
\newtheorem{cor}[prop]{Corollary}
\newtheorem{lem}[prop]{Lemma}

\newtheorem*{convention*}{Convention}
\newtheorem{thmintro}{Theorem}

\theoremstyle{definition}
\newtheorem*{defn*}{Definition}
\newtheorem{defn}[prop]{Definition}
\newtheorem{rem}[prop]{Remark}

\newtheorem*{scholium*}{Scholium}
\newtheorem{example}[prop]{Example}
\newtheorem*{example*}{Example}
\theoremstyle{remark}
\newtheorem*{claim*}{Claim}
\newcommand{\ro}{\varrho}
\newcommand{\se}{\subseteq}
\newcommand{\lra}{\longrightarrow}
\newcommand{\hooklongrightarrow}{\lhook\joinrel\longrightarrow}

\newcommand{\twoheadlongrightarrow}{\relbar\joinrel\twoheadrightarrow}

\DeclareMathOperator{\sig}{\epsilon}
\newcommand{\inv}{^{-1}}

\newcommand{\sL}{\mathscr{L}}

\newcommand{\NN}{\mathbf{N}}

\newcommand{\RR}{\mathbf{R}}
\newcommand{\RRo}{\mathbf{R}_{\sig}}
\newcommand{\ZZ}{\mathbf{Z}}
\newcommand{\pr}{\mathbf{P}}
\newcommand{\dr}{\mathbf{F}}
\newcommand{\odr}{\mathbf{F}^{\mathrm{or}}}
\DeclareMathOperator{\ori}{Or}
\newcommand{\eu}{\mathcal{E}}
\newcommand{\eub}{\mathcal{E}_\mathrm{b}}
\newcommand{\coc}{A}
\newcommand{\coco}{A^{\mathrm{or}}}
\newcommand{\pcoc}{P\!A}
\newcommand{\sul}{E_{\mathrm{Sullivan}}}
\newcommand{\smi}{E_{\mathrm{Smillie}}}
\newcommand{\ssul}{E_{\mathrm{Sullivan}}^{\mathrm{simpl}}}
\newcommand{\ssmi}{E_{\mathrm{Smillie}}^{\mathrm{simpl}}}
\newcommand{\itu}{{IT}}
\newcommand{\SL}{\mathbf{SL}}
\newcommand{\GL}{\mathbf{GL}}
\newcommand{\PSL}{\mathbf{PSL}}
\newcommand{\PGL}{\mathbf{PGL}}
\newcommand{\lw}{L^\infty_\mathrm{w*}}
\newcommand{\mes}{\sL^\infty}
\newcommand{\mesw}{\sL^\infty_\mathrm{w*}}
\newcommand{\bor}{\mathrm{Bor}}
\newcommand{\hb}{\mathrm{H}_\mathrm{b}}
\newcommand{\hh}{\mathrm{H}}
\newcommand{\hhb}{\mathrm{H}_\mathrm{(b)}}
\newcommand{\bu}{\bullet}
\newcommand{\bdd}{_\mathrm{b}}
\begin{document}
\title{The norm of the Euler class}
\author{Michelle Bucher}
\address{Universit\' e de Gen\`eve, 1211 Gen\`eve, Switzerland}
\author{Nicolas Monod}
\address{EPFL, 1015 Lausanne, Switzerland}
\begin{abstract}
We prove that the norm of the Euler class~$\eu$ for flat vector bundles is $2^{-n}$ (in even dimension $n$, since it vanishes in odd dimension). This shows that the Sullivan--Smillie bound considered by Gromov and Ivanov--Turaev is sharp. We construct a new cocycle representing $\eu$ and taking only the two values~$\pm 2^{-n}$; a null-set obstruction prevents any cocycle from existing on the projective space. We establish the uniqueness of an antisymmetric representative for $\eu$ in bounded cohomology.
\end{abstract}
\thanks{Supported in part by the Swiss National Science Foundation}
\maketitle

\section{Introduction}


Let $G$ be a topological group and $\beta\in \hh^\bu(G, \RR)$ a cohomology class. While $\hh^\bu$ denotes the general (``continuous'') cohomology of topological groups (see e.g.~\cite{Wigner73}), we shall mostly be interested in the case where $G$ is a Lie group and $\beta$ corresponds to a characteristic class.

The \textbf{norm} $\|\beta\|$ is by definition the infimum of the sup-norms of all cocycles representing $\beta$ in the classical bar-resolution; thus
$$\| \beta \| = \inf_{f\in\beta}  \|f\|_\infty \in \RR_{\geq 0} \cup \{+ \infty \}$$
(which does not depend on any particular variant of the bar-resolution: homogeneous, inhomogeneous, measurable, smooth, etc.).

\medskip

This norm was introduced by Gromov in~\cite{Gromov} and has important applications since it gives \emph{a priori}-bounds for characteristic numbers; for instance, this explains Milnor--Wood inequalities and in that sense refers back to Milnor~\cite{Milnor58}, compare also~\cite{Wood71, Dupont, Gromov, Bucher-Gelander08, Bucher-GelanderARX}. Further motivations to study this norm come from the Hirzebruch--Thurston--Gromov proportionality principles~\cite{Hirzebruch58,Thurston_unpublished,Gromov} and from the relation to the minimal volume of manifolds \emph{via} the simplicial volume~\cite{Gromov}.

\medskip

However, the norm of only very few cohomology classes is known to this day: the K\"ahler class of Hermitian symmetric spaces in degree two~\cite{Domic-Toledo, Clerc-Orsted03}, the Euler class in $\GL_2^+(\RR)\times \GL_2^+(\RR)$ in degree four~\cite{Bucher08}, and the volume form of hyperbolic $n$-space (in top-degree~$n$)~\cite{Gromov, Thurston_unpublished}, though the latter norm is only explicit in low dimension. In this article, we obtain the norm of the Euler class of flat vector bundles, which was known only for~$n=2$:

\begin{thmintro}\label{thm:norm}
Let $\eu$ be the Euler class in $\hh^n(\GL_n^+(\RR), \RR)$, $n$ even.

Then $\|\eu\| = 2^{-n}$.
\end{thmintro}

More precisely, the (real) Euler class of flat bundles is usually considered as an element in $\hh^n(\GL_n^+(\RR)^\delta, \RR)$, where $\GL_n^+(\RR)^\delta$ is the structure group endowed with the \emph{discrete} topology (so that $\hh^\bu$ reduces to ordinary Eilenberg--MacLane cohomology). There is a unique ``continuous'' class $\eu \in \hh^n(\GL_n^+(\RR), \RR)$ mapping to that ``discrete'' class and it has the same norm (as follows e.g. from the existence of cocompact lattices, by transfer).

\medskip

Based on a simplicial cocycle by Sullivan and Smillie ~\cite{Sullivan76, Smillie_unpublished}, Ivanov and Turaev obtained the upper bound of $\|\eu\| \leq 2^{-n}$ by exhibiting a cocycle with precisely this sup-norm~\cite{Ivanov-Turaev}. By definition, any cocycle provides an upper bound. It is much more difficult to obtain lower bounds because there is no known general method to control the \emph{bounded} coboundaries by which equivalent cocycles may differ, except in degree two, where the double ergodicity of Poisson boundaries leads to resolutions without any $2$-coboundaries~\cite{Burger-Monod1,Burger-Monod3}.

We decompose the lower bound problem into two parts:

\medskip

\textbf{(i)}~\itshape The norm $\|\beta\|$ is equivalently defined as the infimum over all pre-images $\beta\bdd$ in bounded cohomology $\hb^\bu$ of the corresponding semi-norm~$\|\beta\bdd\|$. Can one find an optimal representative~$\beta\bdd$?\upshape

\smallskip

\textbf{(ii)}~\itshape Compute the semi-norm $\|\beta\bdd\|$.\upshape

\bigskip

Concerning point~(i), there can in general be an infinite-dimensional space of pre-images $\beta_b$ for $\beta$. Even for the case at hand, it is not known whether $\eu$ admits a unique pre-image, and indeed the space $\hb^n(\GL^+_n(\RR), \RR)$ has not yet been determined (bounded cohomology remains largely elusive). We shall circumvent this difficulty by using that the Euler class of an oriented vector bundle is \emph{antisymmetric} in the sense that an orientation-reversal changes its sign. Here is the corresponding re-phrasing for the class $\eu$ in group cohomology:

Since inner automorphisms act trivially on cohomology, the canonical action of $\GL_n(\RR)$ upon $\hh^\bu(\GL_n^+(\RR), \RR)$ factors through the order-two quotient group $\GL_n(\RR)/\GL_n^+(\RR)$ (recalling that $n$ is even). Accordingly, we have a canonical decomposition of $\hh^\bu(\GL_n^+(\RR), \RR)$ into eigenspaces for the eigenvalues $1, -1$. Any class in those eigenspaces will be called \textbf{symmetric}, respectively \textbf{antisymmetric}; thus $\eu$ is an example of the latter. The same discussion applies to the bounded cohomology $\hb^\bu(\GL_n^+(\RR), \RR)$. Now we address~(i) using also a result from~\cite{MonodMRL}:

\begin{thmintro}\label{thm:eub}
Let $n$ be even. The space of antisymmetric classes in $\hb^n(\GL_n^+(\RR), \RR)$ is one-dimensional. In particular, there exists a unique antisymmetric class $\eub$ in $\hb^n(\GL_n^+(\RR), \RR)$ whose image in $\hh^n(\GL_n^+(\RR), \RR)$ is the Euler class~$\eu$. Moreover, $\|\eub\|=\|\eu\|$.
\end{thmintro}

\noindent
(This solves Problem~D in~\cite{MonodICM}.)

\begin{defn*}
We call $\eub$ the \textbf{bounded Euler class} of $\GL_n^+(\RR)$. Since the inclusion $\SL_n(\RR)\to \GL_n^+(\RR)$ and quotient $\GL_n^+(\RR)\to \PSL_n(\RR)$ both induce isometric isomorphisms in bounded cohomology we use the same notation~$\eub$ and refer to the bounded Euler class of $\SL_n(\RR)$ and $\PSL_n(\RR)$.
\end{defn*}

Despite its uniqueness with respect to $\GL_n^+(\RR)$, the existence of a canonical bounded class should allow for a finer analysis than the usual class~$\eu$. Indeed, the pull-back of $\eu$ to another group, for instance through a holonomy representation, can admit many more bounded representatives. This type of phenomenon is illustrated in~\cite{Ghys84,Burger-IozziPSUpq}.

\bigskip

We now turn to point~(ii), which is the most substantial part of this article: to compute~$\|\eub\|$. General considerations show that $\eub$ is given by a unique $L^\infty$-cocycle on the projective space. However, although the norm of this unique cocycle is patently~$2^{-n}$, this will not \emph{a priori} give any lower bound on the semi-norm of~$\eub$. Indeed, the isomorphisms given by homotopic resolutions have no reason to be isometric. In fact, to our knowledge, the only general method that guarantees isometries is the use of averaging techniques over amenable groups or actions.

Therefore, we pull back the cocycle to the Grassmannian of complete flags, which is an amenable space and hence computes the right semi-norm. Of course, this comes at the cost of losing the uniqueness of the cocycle since this space is much larger than the projective space and thus supports many coboundaries. We shall nevertheless exhibit a special locus of complete flags where every coboundary must vanish (Section~\ref{sec:coboundaries}). Yet this locus is small; it is a null-set. At this point, we encounter an interesting surprise: The unique $L^\infty$-cocycle that we pulled back \emph{cannot} be represented by an actual cocycle on the projective space when $n\geq 4$; there is an obstruction on another null-set (Proposition~\ref{prop:dPE_neq0}).

\medskip

Nonetheless, on the space of flags, or better of oriented flags, we can remove the obstruction on the blown-up singular locus by a careful iterative deformation. We thus construct an explicit cocycle on oriented flags which, generically, depends only on the projective point (flagstaff) and thus still represents the a.e.\ defined cocycle (Section~\ref{sec:coc}). As desired, this new cocycle is particularly neat even on singular loci:

\begin{thmintro}\label{thm:two-values}
The Euler class $\eu$ of $\GL_n^+(\RR)$ can be represented by an invariant Borel cocycle on $(\GL_n^+(\RR))^{n+1}$ taking only the two values~$\pm2^{-n}$.

(This cocycle is an explicit, algebraically defined invariant on the space of complete oriented flags in~$\RR^n$.)
\end{thmintro}

The existence of some measurable cocycle taking only a finite number of values and representing $\eu$ was expected from~\cite{BucherKarlsson,Bucher07}. Indeed, the corresponding statement was established for the discrete group $\GL_n^+(\RR)^\delta$ and more generally for any primary characteristic class of flat $G$-bundles, whenever $G$ is an algebraic subgroup of $\GL_n(\RR)$. The corresponding statement for the standard topology follows from the proof given there.

\medskip

Finally, we note that our new cocycle is a singular extension of the simplicial cocycles constructed by Sullivan and Smillie~\cite{Sullivan76, Smillie_unpublished}.  More precisely, for any flat bundle over a simplicial complex $K$, the classifying map $|K|\rightarrow B \GL_n^+(\RR)^\delta$ can be chosen so that the pull-back of our cocycle is precisely Smillie's simplicial cocycle when restricted to the simplices of $K$. It presents the advantage of being immediate to evaluate, in contrast to the Ivanov--Turaev cocycle~\cite{Ivanov-Turaev} which is obtained by taking averages of Sullivan--Smillie cocycles. Moreover, as it is defined on all singular simplices simultaneously, and not only the simplices of a given triangulation (or of one particular representative of the fundamental cycle) like the simplicial cocycles of Sullivan--Smillie, it might be more useful for actually computing Euler numbers of flat bundles over manifolds whose triangulations are often very complicated, if known at all.

\section{General notation}

Throughout the paper, $n$ is an even integer.

\medskip
We agree that a basis of a finite-dimensional vector space is an \emph{ordered} tuple $(v_1, \ldots, v_k)$. It thus endows the space with an orientation. If the vectors $v_1, \ldots, v_k$ are merely linearly independent, we denote by $\langle v_1, \ldots, v_k\rangle$ the oriented space that they span. When confusion is unlikely, we use the same notation for an oriented space and its underlying vector space. There is a natural direct sum $V\oplus W$ of oriented spaces $V, W$; the orientation can depend on the order of summands. By default, $\RR^k$ is endowed with its canonical basis $(e_1, \ldots, e_k)$ and with the corresponding orientation. We write $e_0=e_1+\ldots+e_k$.

If $V$ denotes the vector space $\RR^k$ endowed with some orientation, let $\ori(V)\in\{-1,1\}$ be the sign of this orientation relatively to the canonical orientation. Further, if $(v_{1},\ldots,v_{k})$ is a basis of $\RR^{k}$, we write $\ori(v_{1},\ldots,v_{k})$ for $\ori(\langle v_{1},\ldots,v_{k} \rangle)$ and extend $\ori$ to a function on all $k$-tuples of elements in $\RR^{k}$ by setting $\ori(v_{1},\ldots,v_{k})=0$ if $(v_{1},\ldots,v_{k})$ is not a basis.

We write $\sig(x)\in\{-1,1\}$ for the sign of $x\in \RR^*$ and extend it to a homomorphism on $\GL_n(\RR)$ as the sign of the determinant; $\GL_n^+(\RR)$ is its kernel. Notice that $\sig$ descends to $\PGL_n(\RR)$ since $n$ is even. We denote by $\RRo$ the $\GL_n(\RR)$-module (or $\PGL_n(\RR)$-module) $\RR$ endowed with multiplication by $\sig$.

Given any $(k+1)$-tuple $(x_0, \ldots, x_k)$, the $k$-tuple obtained by dropping $x_i$ is written $(x_0, \ldots, \widehat{x_i}, \ldots, x_k)$. Cocycles and coboundaries in various function spaces will be with respect to the differential $d=\sum_{i=0}^k (-1)^i d_i$, where $d_i$ is the evaluation on $(x_0, \ldots, \widehat{x_i}, \ldots, x_k)$.

The projective space is denoted by $\pr(\RR^n)$; we often use the same notation for both elements in $\RR^n$ and their image in $\pr(\RR^n)$.

\medskip

We refer to~\cite{Burger-Monod3,Monod} for background on the bounded cohomology of locally compact groups and to~\cite{BucherKarlsson,Bucher07,Gromov} for the relation to characteristic classes.

\section{The almost-cocycle on the projective space}\label{sec:pcoc}

The bounded cohomology of $\GL_n(\RR)$ with coefficients in $\RRo$ can be represented by $L^\infty$-cocycles on the projective space for reasons that we shall explain in Section~\ref{sec:eub}. Therefore, we begin with a few elementary observations on equivariant functions on the projective space.

\begin{prop}\label{prop:unique:pcoc}
There is, up to scaling, a unique non-zero $\GL_n(\RR)$-equivariant map
$$(\pr(\RR^{n}))^{q}\longrightarrow\RRo$$
for $q=n+1$; there is none for $q\leq n$.
\end{prop}

With the right scaling, the unique map above will be seen to yield an $L^\infty$-cocycle representing the Euler class. Interestingly, this a.e.\ function class cannot be given by an actual cocycle:

\begin{prop}\label{prop:dPE_neq0}
The coboundary of a non-zero $\GL_n(\RR)$-equivariant map
$$(\pr(\RR^{n}))^{n+1}\lra\RRo$$
does not vanish everywhere on~$(\pr(\RR^{n}))^{n+2}$ unless $n=2$.
\end{prop}

The proof of the above propositions is an occasion to introduce a concept that will be used throughout:

\begin{defn}
Let $k\geq n$. A $k$-tuple in $\RR^n$ or in $\pr(\RR^n)$ is \textbf{hereditarily spanning} if every subcollection of $n$ elements spans $\RR^n$.
\end{defn}

\begin{example}\label{ex:spanning}
The $(n+1)$-tuple $(x, e_1, \ldots, e_n)$ is hereditarily spanning if and only if all coordinates of $x$ are non-zero. The $(n+2)$-tuple $(e_0, e_1, \ldots, e_n,x)$ is hereditarily spanning if and only if all coordinates of $x$ are non-zero and distinct.
\end{example}

\begin{proof}[Proof of Proposition~\ref{prop:unique:pcoc}]
The action of $\PGL_n(\RR)$ on hereditarily spanning $(n+1)$-tuples in $\pr(\RR^n)$ is free and transitive (as is apparent by e.g. considering Example~\ref{ex:spanning}). This implies existence, choosing the value zero on all other $(n+1)$-tuples. Next, we claim that in fact any $\GL_n(\RR)$-equivariant map $f$ must vanish on tuples $(x_0, \ldots, x_n)$ that are not hereditarily spanning; this entails uniqueness.

To prove the claim, we can assume by symmetry that $x_1, \ldots, x_n$ are contained in a subspace $V\se \RR^n$ of dimension $n-1$. By $\GL_n(\RR)$-equivariance, we can further assume that $x_0$ is either perpendicular to $V$ or contained in it. Let now $g$ be the orthogonal reflection along $V$; then $\epsilon(g)=-1$ and $g$ fixes the projective points $x_0,x_1, \ldots, x_n$. Therefore $f$ vanishes at that tuple, as claimed. The argument given for this claim also settles the case~$q\leq n$.
\end{proof}

\begin{rem}
Had we allowed $n$ to be odd, there would be no non-zero $\GL_n(\RR)$-equivariant map $(\pr(\RR^{n}))^{q}\to\RRo$ for any $q$ whatsoever since then the centre of $\GL_n(\RR)$, which acts trivially on $\pr(\RR^{n})$, contains elements with negative determinant (and this is the underlying reason for the vanishing of the Euler class). Consider $\GL_n^+(\RR)$-invariant maps instead; one then finds that $\GL_n^+(\RR)$ has only one orbit of hereditarily spanning $(n+1)$-tuples whereas it has two when $n$ is even.
\end{rem}

\begin{proof}[Proof of Proposition~\ref{prop:dPE_neq0}]
Let $f$ be a map as in the statement; for simpler notation, we consider $f$ as defined on $(\RR^n)^{n+1}$. Let us evaluate $df$ at the $(n+2)$-tuple $(e_0, e_1, \ldots, e_n, e_1+e_2)$. We examine all sub-$(n+1)$-tuples occurring in the evaluation of $df$:

First, $f$ vanishes on $(e_1, \ldots, e_n, e_1+e_2)$ since it is not hereditarily spanning as soon as $n> 2$ (Example~\ref{ex:spanning}). Next, one checks that $(e_0, \ldots, \widehat{e_i},\ldots, e_n, e_1+e_2)$ is not hereditarily spanning whenever $1\leq i\leq n$ (distinguishing cases as $1\leq i\leq 2$ or $i>2$), hence $f$ vanishes there aswell. However, $f$ is non-zero on $(e_0, e_1, \ldots e_n)$ since it belongs to the hereditarily spanning orbit; this establishes the claim.
\end{proof}

The existence and uniqueness proof indicates of course exactly what the equivariant map is; nevertheless, we wish to record an explicit formula. Define first the function
\[\pcoc:(\RR^{n})^{n+1}\longrightarrow\{-1,0,1\}, \kern5mm \pcoc(v_{0},\ldots,v_{n})=\prod_{i=0}^{n}\ori(v_{0},\ldots,\widehat{v_{i}},\ldots,v_{n}).\]
Since $n$ is even and $\ori(v_{0},\ldots, \widehat{v_{i}},\ldots, \lambda v_j,\ldots, v_{n}) = \sig(\lambda) \ori(v_{0},\ldots, \widehat{v_{i}},\ldots,v_{n})$ for all $\lambda\in\RR^*$ and $j\neq i$, we deduce:

\begin{lem}\label{lem:pcoc}
$\pcoc$ descends to an alternating equivariant map
\[\pcoc:(\pr(\RR^{n}))^{n+1}\longrightarrow\{-1,0,1\}\]
(denoted by the same symbol).\qed
\end{lem}

One can check explicitly that this map is an a.e.\ cocycle; more precisely:

\begin{prop}\label{prop:dPE=0}
Let $v_{0},\ldots,v_{n+1}\in\RR^{n}$ be hereditarily spanning. Then
\[
d \pcoc(v_{0},\ldots,v_{n+1})=0.\]
\end{prop}

\begin{proof}[Explicit Proof]
Using transitivity properties, it suffices to consider $(n+2)$-tuples $v_i$ of the form $(e_0, e_1, \ldots, e_n, x)$; the coordinates of $x$ are non-zero and distinct. Moreover, applying monomial matrices we can assume that they are arranged in increasing order; this might permute $e_1, \ldots, e_n$ but we can rearrange the latter since $\pcoc$ is alternating. Let thus $k\in\{0, \ldots, n\}$ be such that $x_{1}<x_{2}<\ldots<x_{k}<0<x_{k+1}<\ldots<x_{n}$. One now checks
\[
\ori(e_{1},\ldots,\widehat{e_{j}},\ldots,e_{n},x)=(-1)^{j}\cdot\sig(x_{j})=
\begin{cases}
(-1)^{j+1} &\text{if $1\leq j\leq k$,}\\
(-1)^j &\text{if $k<j\leq n$,}
\end{cases}
\]
\[
\ori(e_{0},\ldots,\widehat{e_{i}},\ldots,\widehat{e_{j}},\ldots,e_{n},x)=(-1)^{i+j+1} \kern5mm 1\leq i<j\leq n,
\]
\[
\ori(e_{0},e_{1},\ldots,\widehat{e_{j}},\ldots,e_{n})=(-1)^{j+1} \kern5mm 1\leq j\leq n.
\]
We can thus compute
\begin{align*}
\pcoc(e_{1},\ldots,e_{n},x) &= (-1)^{n/2} (-1)^k,\\
\pcoc(e_{0},e_{1},\ldots,\widehat{e_{i}},\ldots,e_{n},x) &= (-1)^i \sig(x_i) \Bigg(\prod_{j=1}^{i-1}(-1)^{i+j+1}\Bigg)\Bigg(\prod_{j=i+1}^{n}(-1)^{i+j+1}\Bigg) (-1)^{i+1}\\
&= (-1)^{n/2}\sig(x_i), \kern5mm (\text{here $1\leq i\leq n$})\\
\pcoc(e_{0},e_{1},\ldots,e_{n}) &= \prod_{j=1}^{n}(-1)^{j+1}=(-1)^{n/2}.
\end{align*}
The cocycle relation becomes
\begin{align*}
d \pcoc(e_{0},e_{1},\ldots,e_{n},x) &= \sum_{i=0}^{n}(-1)^{i}\pcoc(e_{0},\ldots,\widehat{e_{i}},\ldots,e_{n},x)-\pcoc(e_{0},\ldots,e_{n})\\
&= (-1)^{n/2}\left[ (-1)^{k} - \sum_{i=1}^{k}(-1)^{i} + \sum_{i=k+1}^{n}(-1)^{i} - 1\right]
\end{align*}
which vanishes indeed; we used throughout that $n$ is even.
\end{proof}

The proposition can also be derived without any computation if one uses the (independent) fact that there has to be \emph{some} $L^\infty$-cocycle, as follows from the boundedness of the Euler class in light of arguments given in the proof of Theorem~\ref{thm:coho:proj}.

\begin{proof}[Alternate proof of Proposition~\ref{prop:dPE=0}]
The sets $H_k$ of hereditarily spanning $k$-tuples are open dense in $(\RR^n)^k$ and preserved under omitting variables as long as at least $n$ variables are left. Therefore, since $\ori$ is locally constant on $H_n$, we deduce that $\pcoc$ and $d\pcoc$ are locally constant on $H_{n+1}$ and $H_{n+2}$. However, if we know that $d\pcoc$ vanishes almost everywhere, it now follows that it vanishes everywhere on $H_{n+2}$.
\end{proof}

Yet another viewpoint will emerge in Section~\ref{sec:ss}.

\section{A cocycle on the flag space}\label{sec:coc}

We have seen that $\pcoc$ cannot be promoted to be a true cocycle on the projective space. We shall remedy this situation by blowing up the singular (non-hereditarily-spanning) locus and working with complete oriented flags. By an iterative deformation construction, this leads to a cocycle~$\coco$ in Theorem~\ref{thm:dE=0} below. An added benefit is that our modified cocycle $\coco$ will take only the values~$\pm 1$. We then deflate this cocycle to the usual flag space, still keeping the same values as $\pcoc$ on the hereditarily spanning tuples.

\bigskip

Denote by $\dr(\RR^{n})$ the set of complete flags $F$ in $\RR^{n}$,
\[
F: \ F^0=\{0\}\subset F^{1}\subset\ldots\subset F^{n-1}\subset F^{n}=\RR^{n},\]
where each $F^{i}$ is an $i$-dimensional subspace of $\RR^{n}$. The set $\odr(\RR^{n})$ of complete \emph{oriented} flags consists of complete flags $F$ where each $F^{i}$ is furthermore endowed with an orientation. Equivalently, each $F^{i}$ is given together with the choice of an open half space $(F^{i})^{+}$ bounded by $F^{i-1}$. The positive orientation on $F^{i}$ will then be determined by any basis $(v_{1},\ldots,v_{i-1},x)$, where $(v_{1},\ldots,v_{i-1})$ is a positively oriented basis of $F^{i-1}$ and $x\in(F^{i})^{+}$. Note that $\odr(\RR^{n})$ is a $2^{n}$-cover of $\dr(\RR^{n})$.

Let $k\in\{0,1,\ldots,n-1\}$, let $W$ be a $k$-dimensional oriented subspace of $\RR^{n}$ and let $F\in\odr(\RR^{n})$ be a complete oriented flag. Define a $(k+1)$-dimensional oriented subspace $\left[W,F\right]$ of $\RR^{n}$ as follows: Let $d$ be the unique integer $1\leq d\leq k+1$ with
\[
F^{d-1}\subset W\;\mbox{and}\; F^{d}\nsubseteq W.\]
Define $[W,F]$ to be the vector space generated by $W$ and $F^{d}$, endowed with the orientation given by $(w_{1},\ldots,w_{k},x)$, where $(w_{1},\ldots,w_{k})$ is a positively oriented basis of $W$ and $x\in(F^{d})^{+}$.

Given complete oriented flags $F_{1},\ldots,F_{k}\in\odr(\RR^{n})$, we define a $k$-dimensional oriented vector space $[F_{1},\ldots,F_{k}]$ inductively as follows: For $k=1$, let $[F_{1}]=F_{1}^{1}$. For $k>1$, let $[F_{1},\ldots,F_{k}]=[[F_{1},\ldots,F_{k-1}],F_{k}]$.

\begin{rem}\label{rem:[]}
If the lines $F_i^1$ are linearly independent, then simply $[F_1, \ldots, F_k] = \langle F_1^1, \ldots, F_k^1\rangle$. At the other extreme, if all $F_i$ are the same oriented flag $F\in\odr(\RR^{n})$, then $[F,\ldots,F]=F^{k}$.
\end{rem}

Define a function
\[
\coco:(\odr(\RR^{n}))^{n+1}\longrightarrow\{-1,1\}, \kern5mm \coco(F_{0},\ldots,F_{n})=\prod_{i=0}^{n}\ori([F_{0},\ldots,\widehat{F_{i}},\ldots,F_{n}]).\]

\begin{lem}
$\coco$ is $\GL_n(\RR)$-equivariant.
\end{lem}

\begin{proof}
This follows from $[gF_{1},\ldots,gF_{n}] = g[F_{1},\ldots,F_{n}]$ and $\ori(g[F_{1},\ldots,F_{n}]) = \sig(g)\cdot\ori([F_{1},\ldots,F_{n}])$ for $g\in\GL_n(\RR)$.
\end{proof}

\begin{thm}\label{thm:dE=0}
$d \coco(F_{0},\ldots,F_{n+1})=0$ for any $F_{0},\ldots,F_{n+1}\in\odr(\RR^{n})$.
\end{thm}

\begin{lem}\label{lemma:Or(V,x)=Or(V,F)}
Let $V_{1},\ldots,V_{q}$ be $(n-1)$-dimensional oriented subspaces of $\RR^{n}$, where $q\in\NN$ is arbitrary. For any $F\in\odr(\RR^{n})$ there exists $x\in\RR^{n}\setminus\bigcup_{i=1}^{q}V_{i}$ such that
\[
\ori\left(V_{i}\oplus\langle x\rangle\right)=\ori(\left[V_{i},F\right]),\]
for every $i=1,\ldots,q$. 
\end{lem}

\begin{proof}
Let $x_{1},\ldots,x_{n}\in\RR^{n}$ be a sequence of points $x_{d}\in(F^{d})^{+}$ with the following property: For every $i=1,\ldots,q$, the intersection of $V_{i}$ with the affine segment $[x_{d-1},x_{d}]$ is either empty, equal to $\{x_{d-1}\}$ or to the whole segment. Let us prove by induction that such a sequence exists: For $d=1$, take any $x_{1}\in(F^{1})^{+}$. Suppose that $x_{1},\ldots,x_{d-1}$ have been constructed. Let $U$ be a convex neighbourhood of $x_{d-1}$ such that, for every $i=1,\ldots,q$, if $V_{i}\cap U\neq\emptyset$, then $x_{d-1}\in V_{i}$. Any $x_{d}\in U\cap(F^{d})^{+}$ will work. 

To prove the lemma, it suffices to take $x=x_{n}$. Indeed, for every
$i=1,\ldots,q$, let $d_{i}$ be such that $F^{d_{i}-1}\subset V_{i}$
and $F^{d_{i}}\nsubseteq V_{i}$. Then, by definition, for any $y\in\left(F^{d_{i}}\right)^{+}$,
and in particular for $x_{d_{i}}\in(F^{d_{i}})^{+}$
\[
\ori([V_{i},F])=\ori(V_{i}\oplus\langle y\rangle)=\ori(V_{i}\oplus\langle x_{d_{i}}\rangle).\]
As $x_{d_{i}}\notin V_{i}$, the points $x_{d_{i}},x_{d_{i}+1},\ldots,x_{n}$
do by construction all lie on the same half space with respect to $V_{i}$,
so that
\[
\ori(V_{i}\oplus\langle x_{d_{i}}\rangle)=\ori(V_{i}\oplus\langle x_{d_{i}+1}\rangle)=\ldots=\ori(V_{i}\oplus\langle x_{n}\rangle),\]
which finishes the proof of the lemma. 
\end{proof}

Observe that $x$ could be any point in the same connected component of ${\RR^{n}\setminus\bigcup_{i} V_{i}}$ as $x_{n}$. This
is the unique connected component $C$ such that the intersection $\overline{C}\cap(F^{d})^{+}$ is non-empty for every $d=1,2,\ldots,n$.

\begin{prop}\label{prop:Fi->xi}
Let $F_{0},\ldots,F_{n+1}\in\odr(\RR^{n})$ be complete oriented flags. There exist $x_{0},\ldots,x_{n+1}\in\RR^{n}$ such that
\[
\ori([F_{0},\ldots,\widehat{F_{i}},\ldots,\widehat{F_{j}},\ldots,F_{n+1}]) = \ori(x_{0},\ldots,\widehat{x_{i}},\ldots,\widehat{x_{j}},\ldots,x_{n+1})
\]
for every $0\leq i<j\leq n+1$. 
\end{prop}

\begin{proof}
We will prove the following claim by downwards induction on $k$, starting with $k=n$ and going down to $k=-1$. The latter case proves the proposition.

\begin{claim*}
There exist $x_{k+1},\ldots,x_{n+1}$ such that $\ori([F_{0},\ldots,\widehat{F_{i}},\ldots,\widehat{F_{j}},\ldots ,F_{n+1}])$ equals
\begin{align*}
\ori([F_{0},\ldots,\widehat{F_{i}},\ldots,\widehat{F_{j}},\ldots,F_{k}]\oplus\langle x_{k+1},\ldots,x_{n+1}\rangle)\ & \mbox{ for }0\leq i<j\leq k,\\
\ori([F_{0},\ldots,\widehat{F_{i}},\ldots,F_{k}]\oplus\langle x_{k+1},\ldots,\widehat{x_{j}},\ldots,x_{n+1}\rangle)\ & \mbox{ for }0\leq i\leq k<j\leq n+1,\\
\ori([F_{0},\ldots,F_{k}]\oplus\langle x_{k+1},\ldots,\widehat{x_{i}},\ldots,\widehat{x_{j}},\ldots,x_{n+1}\rangle)\ & \mbox{ for }k+1\leq i<j\leq n+1.
\end{align*}
\end{claim*}

\noindent\emph{Proof of the claim}. For the case $k=n$, we apply Lemma~\ref{lemma:Or(V,x)=Or(V,F)} to the family of oriented $(n-1)$-dimensional subspaces
$$[F_{0},\ldots,\widehat{F_{i}},\ldots,\widehat{F_{j}},\ldots,F_{n}] \kern1cm (i<j\leq n)$$
together with the oriented flag $F_{n+1}$ to find $x_{n+1}\in\RR^{n}$ such that
\[\ori([F_{0},\ldots,\widehat{F_{i}},\ldots,\widehat{F_{j}},\ldots,F_{n}]\oplus\left\langle x_{n+1}\right\rangle )=\ori([F_{0},\ldots,\widehat{F_{i}},\ldots,\widehat{F_{j}},\ldots,F_{n},F_{n+1}]).\]
Next, we suppose inductively that the claim is true for $k$ and establish it for ${k-1}$. The inductive assumption implies in particular that none of $x_{k+1}, \ldots, x_{n+1}$ belongs to the subspace $[F_{0},\ldots,\widehat{F_{i}},\ldots,\widehat{F_{j}},\ldots,F_{k}]$ in case $i<j\leq k$, and a similar statement for the subspaces $[F_{0},\ldots,\widehat{F_{i}},\ldots,F_{k}]$ and $[F_{0},\ldots,F_{k}]$ in the two other cases.
Let $V_{ij}$ denote the oriented $(n-1)$-dimensional subspaces
\begin{align*}
[F_{0},\ldots,\widehat{F_{i}},\ldots,\widehat{F_{j}},\ldots,F_{k-1}]\oplus\langle x_{k+1},\ldots,x_{n+1}\rangle, & \mbox{ for }0\leq i<j\leq k-1,\\
[F_{0},\ldots,\widehat{F_{i}},\ldots,F_{k-1}]\oplus\langle x_{k+1},\ldots,\widehat{x_{j}},\ldots,x_{n+1}\rangle, & \mbox{ for \ensuremath{0\leq i\leq k-1,\mbox{ }k+1\leq j\leq n+1,}}\\
[F_{0},..,F_{k-1}]\oplus\langle x_{k+1},\ldots,\widehat{x_{i}},\ldots,\widehat{x_{j}},\ldots,x_{n+1}\rangle, & \mbox{ for }k+1\leq i<j\leq n+1.
\end{align*}
Apply Lemma~\ref{lemma:Or(V,x)=Or(V,F)} to the subspaces
$V_{ij}$ and the oriented flag $F_{k}$ to find $x_{k}$ such that
\[
\ori(V_{ij}\oplus\left\langle x_{k}\right\rangle )=\ori(\left[V_{ij},F_{k}\right]).\]
We now have, for $0\leq i<j\leq k-1$,
\begin{align*}
&\ori([F_{0},\ldots,\widehat{F_{i}},\ldots,\widehat{F_{j}},\ldots,F_{k-1}]\oplus\langle x_{k},x_{k+1},\ldots,x_{n+1}\rangle)\\
=\ &(-1)^{n+1-k}\ori(V_{ij}\oplus\langle x_{k}\rangle) = (-1)^{n+1-k}\ori(\left[V_{ij},F_{k}\right])\\
=\ &(-1)^{n+1-k}\ori([[F_{0},\ldots,\widehat{F_{i}},\ldots,\widehat{F_{j}},\ldots,F_{k-1}]\oplus\langle x_{k+1},\ldots,x_{n+1}\rangle,F_{k}])\\
=\ &\ori([F_{0},\ldots,\widehat{F_{i}},\ldots,\widehat{F_{j}},\ldots,F_{k-1},F_{k}]\oplus\langle x_{k+1},\ldots,x_{n+1}\rangle)\\
=\ &\ori([F_{0},\ldots,\widehat{F_{i}},\ldots,\widehat{F_{j}},\ldots ,F_{n+1}]),
\end{align*}
where the last equality is our induction hypothesis. For the penultimate equality, in order to permute $\langle x_{k+1},\ldots,x_{n+1}\rangle$ with $F_{k}$, we have used that $x_{k+1}, \ldots, x_{n+1}$ do not belong to the subspaces $[F_{0},\ldots,\widehat{F_{i}},\ldots,\widehat{F_{j}},\ldots,F_{k}]$; in particular, the relevant $d$-component $F_k^d$ of $F_k$ involved in the definition of $[\ldots, F_k]$ on both sides of that equality remains the same.

\smallskip
The two cases with $j\geq k$ are proved almost identically (a
difference is the sign of the factor $(-1)^{n+1-k}$). We have thus
proved the claim and the proposition.
\end{proof}

\begin{proof}[Proof of Theorem~\ref{thm:dE=0}]
Let $F_{0},\ldots,F_{n+1}$ be oriented flags. By Proposition~\ref{prop:Fi->xi} there exists $x_{0},\ldots,x_{n+1}$ such that
\[ \ori([F_{0},\ldots,\widehat{F_{i}},\ldots,\widehat{F_{j}},\ldots,F_{n+1}])=\ori(x_{0},\ldots,\widehat{x_{i}},\ldots,\widehat{x_{j}},\ldots,x_{n+1})\]
for every $0\leq i<j\leq n+1$. In particular, $x_{0},\ldots,x_{n+1}$ is hereditarily spanning and furthermore
\[ \coco(F_{0},\ldots,\widehat{F_{i}},\ldots,F_{n+1})=\pcoc(x_{0},\ldots,\widehat{x_{i}},\ldots,x_{n+1})\]
for every $0\leq i\leq n+1$. The theorem now follows from the validity of the cocycle relation $d \pcoc=0$ for hereditarily spanning $(n+1)$-tuples proven in Proposition~\ref{prop:dPE=0}.
\end{proof}

Finally, we define the map
$$\coc:(\dr(\RR^{n}))^{n+1}\longrightarrow [-1,1], \kern5mm \coc(F_0, \ldots, F_n) = 2^{-n(n+1)} \sum \coco(F'_0, \ldots, F'_n)$$
where the sum ranges over all oriented flags $F'_i$ having $F_i$ as underlying flag.

\begin{cor}\label{cor:coc}
The map $\coc$ is a cocycle ($d\coc=0$ everywhere) and is $\GL_n(\RR)$-equivariant. Moreover, $\coc(F_0, \ldots, F_n)=\pcoc(F_0^1, \ldots, F_n^1)$ as soon as $(F_0^1, \ldots, F_n^1)$ is hereditarily spanning.
\end{cor}

In other words, denoting by $h: \dr(\RR^n)\to\pr(\RR^n)$ the flagstaff projection $h(F)=F^1$, we have $\coc=h^*\pcoc$ on all $(n+1)$-tuples with hereditarily spanning image in $\pr(\RR^n)^{n+1}$.

\begin{proof}[Proof of Corollary~\ref{cor:coc}]
If $f$ is any function on $(\odr(\RR^{n}))^{p+1}$, $p\geq 0$, we define the deflation $\mathrm{defl}(f)$ on $(\dr(\RR^{n}))^{p+1}$ by the average
$$\mathrm{defl}(f) (F_0, \ldots, F_p) = 2^{-n(p+1)} \sum f(F'_0, \ldots, F'_p)$$
over all oriented representatives $F'_i$ of the flags $F_i$. Thus, $\coc=\mathrm{defl}(\coco)$. The definition ensures $d \circ\mathrm{defl} = \mathrm{defl}\circ d$ and moreover $\mathrm{defl}$ commutes with the diagonal $\PGL_n(\RR)$-actions. As for the additional claim, it follows from the definition of $\coco$, Remark~\ref{rem:[]} and Lemma~\ref{lem:pcoc}.
\end{proof}

\section{Vanishing of coboundaries}\label{sec:coboundaries}

Given a basis $(w_{1},\ldots,w_{n})$ of $\RR^{n}$, define
$F(w_{1},\ldots,w_{n})\in\dr(\RR^{n})$ to be the complete
flag
\[
\{0\}\subset\langle w_{1}\rangle\subset\langle w_{1},w_{2}\rangle\subset\ldots\subset\langle w_{1},\ldots,w_{i}\rangle\subset\ldots\subset\langle w_{1},\ldots,w_{n}\rangle=\RR^{n}.\]
%

\begin{lem}\label{lem:coboundaries}
Let $F_{0},\ldots,F_{n}\in\dr(\RR^{n})$ be the complete flags
\begin{align*}
F_{0} &= F(e_0,e_{1},\ldots,e_{n-1}),\\
F_{1} &= F(e_{1},e_{2},\ldots,e_{n}),\\
F_{i} &= F(e_{i},e_{i+1},\ldots,e_{n},e_0,\ldots,e_{i-1}),\qquad\mbox{for }2\leq i\leq n.
\end{align*}
If a cochain $b:(\dr(\RR^{n})^{n}\rightarrow\RRo$ is $\PGL_n(\ZZ)$-equivariant, then
\[
d b(F_{0},\ldots,F_{n})=0.\]
\end{lem}

\begin{proof}
We shall show that for each $i=0,1,\ldots,n$, there exists $g_i\in \GL_n(\ZZ)$ with $\mbox{det}(g_i)=-1$ such that $g_i F_{j}=F_{j}$ for every $j\neq i$. The lemma follows since
\[
b(F_{0},\ldots,\widehat{F_{i}},\ldots,F_{n})=-b(g_i F_{0},\ldots,\widehat{g_i F_{i}},\ldots,g_i F_{n})=-b(F_{0},\ldots,\widehat{F_{i}},\ldots,F_{n})\]
by equivariance. Taking indices modulo $n+1$, the matrix $g_i$ is defined so that it fixes $e_{i+1},\ldots,e_{i-2}$, sends $e_{i-1}$ to $-e_{i-1}$ and maps $e_{i}$ to a linear combination $e_{i} \pm 2 e_{i-1}$ of $e_{i-1}$ and $e_i$. These properties guarantee that $g_i$ fixes the flags $F_{j}$ for $j\neq i$ and has determinant~$-1$. Explicitly, $g_i$ is
\[
\left(\begin{array}{cc}
\mbox{Id}_{n-1} & 0\\
0 & -1\end{array}\right),\ 
\left(\begin{array}{cccc}
-1\\
-2 & 1\\
\vdots &  & \ddots\\
-2 &  &  & 1\end{array}\right),\ 
\left(\begin{array}{cccc}
\mbox{Id}_{i-2}\\
& -1 & 2\\
& 0 & 1\\
&  &  & \mbox{Id}_{n-i}\end{array}\right)\]
for respectively $i=0$, $i=1$ and $2\leq i\leq n$.
\end{proof}

\section{Functoriality and the semi-norm}
\label{sec:funct}
In this section, we compare two ways to define a bounded cohomology class using the cocycle~$\coc$. To keep track of the distinction, we use $q$ to denote the usually tacit map associating an a.e.\ function class to a function. The first way is to consider the cocycle $q\coc$ in the resolution
$$0 \lra \RRo \lra L^\infty(\dr(\RR^n), \RRo) \lra L^\infty(\dr(\RR^n)^2, \RRo)\lra\cdots$$
The cohomology of the (non-augmented) complex of invariants of this resolution is canonically isometrically isomorphic to $\hb^\bu(\GL_n(\RR), \RRo)$ thanks to the amenability of the action on $\dr(\RR^n)$, see e.g.~\cite[Thm.~2]{Burger-Monod3}. (We recall that $\hb^\bu(\GL_n(\RR), \RRo)$ is endowed with a canonical infimal semi-norm~\cite[7.3.1]{Monod}.) Let $[q\coc]\bdd$ be the corresponding element of $\hb^n(\GL_n(\RR), \RRo)$.

The actual value of the semi-norm of $[q\coc]\bdd$ is not obvious since we have no good understanding of coboundaries up to null-sets in the above resolution. Therefore, we use a second approach, considering $\coc$ as a cocycle on the \emph{set} $\dr(\RR^n)$ so that we can use Section~\ref{sec:coboundaries}. Comparing the two approaches, we shall obtain:

\begin{thm}\label{thm:funct}
The semi-norm of $[q\coc]\bdd$ in $\hb^n(\GL_n(\RR), \RRo)$ is $\|[q\coc]\bdd\|=1$.
\end{thm}

\begin{rem}
We shall deduce the proof from a more general discussion because it might be useful for the study of characteristic classes of other Lie groups. In the special case of $\GL_n(\RR)$, a minor simplification would be available because the stabiliser of a complete flag is amenable as abstract group, whilst in general minimal parabolics are only amenable as topological groups. This accounts for our explicit use of a lattice~$\Gamma$, whilst for $\GL_n(\RR)$ one could instead work over the discrete group $\GL_n(\RR)^\delta$ and only use the existence of a lattice to control indirectly the semi-norm in bounded cohomology for $\GL_n(\RR)^\delta$.
\end{rem}

Let $G$ be a locally compact second countable group, $\Gamma<G$ a lattice and $V$ a coefficient $G$-module (i.e.\ $V$ is the dual of a separable continuous isometric Banach $\Gamma$-module; below, $V=\RRo$). Let $P<G$ be a closed amenable subgroup and endow $G/P$ with its unique $G$-quasi-invariant measure class (see \emph{e.g.}~\cite[\S\,23.8]{Simonnet}). Denote by $\lw(G/P, V)$ the coefficient $G$-module of essentially bounded weak-* measurable function classes; a function $f:G/P\to V$ is weak-* measurable if $u\circ f$ is measurable for any predual vector $u$. Denote by $\mesw(G/P, V)$ the Banach $G$-module of bounded weak-* measurable functions. (Beware that many authors use the notation $\mes$ for \emph{essentially} bounded functions and use the corresponding semi-norm. Of course the two conventions lead to the same quotient $L^\infty$ but the distinction is needed here.) We use further the standard notation $\ell^\infty(G/P, V)$ for the Banach $G$-module of \emph{all} bounded functions $G/P\to V$. All these notations are extended to function (classes) on $(G/P)^{n+1}$, $n\geq 0$. Consider the quotient and inclusion maps
$$q: \mesw \twoheadlongrightarrow \lw, \kern10mm i: \mesw \hooklongrightarrow \ell^\infty.$$
Let now $\omega$ be a cocycle in $\mesw((G/P)^{n+1}, V)^G$. On the one hand, $q\omega$ determines an element $[q\omega]\bdd$ of $\hb^n(G, V)$ whose canonical semi-norm is realized as the infimal $L^\infty$-norm of all cohomologous elements in $\lw((G/P)^{n+1}, V)^G$; this is a special case of~\cite[2.3.2]{Burger-Monod3} or~\cite[7.5.3]{Monod}. On the other hand, we claim that $i\omega$ determines an element $[i\omega_\Gamma]\bdd$ of $\hb^n(\Gamma, V)$ whose canonical semi-norm is realized as the infimal $\ell^\infty$-norm of all cohomologous elements in $\ell^\infty((G/P)^{n+1}, V)^\Gamma$. Indeed, the averaging argument of the above references is stated for locally compact second countable groups with an amenable action on standard measure space; however, in the case of a discrete group, it can be repeated verbatim for any amenable action on any set, since all measurability issues disappear. Therefore, we only need to verify that the $\Gamma$-action on $G/P$ viewed as a set is amenable, which amounts to the amenability of all isotropy groups $\Gamma\cap gPg\inv$ where $g$ ranges over $G$ (this is a degenerate form of Theorem~5.1 in~\cite{Adams-Elliott-Giordano}). The latter group being closed in $gPg\inv$, it is amenable as topological group; being discrete, it is amenable.

\begin{lem}\label{lem:compare}
The image of $[q\omega]\bdd$ under the restriction map $\hb^n(G, V)\to \hb^n(\Gamma, V)$ coincides with $[i\omega_\Gamma]\bdd$.
\end{lem}

\begin{proof}
The restriction can be realized by the inclusion map
$$\mesw((G/P)^{n+1}, V)^G \hooklongrightarrow \mesw((G/P)^{n+1}, V)^\Gamma,$$
see~\cite[8.4.2]{Monod}. Now the lemma follows from the functoriality statements~\cite[7.2.4, 7.2.5]{Monod} applied to the $\Gamma$-resolution $\mesw((G/P)^{n+1}, V)$ in comparison to the two relatively $\Gamma$-injective resolutions $\lw((G/P)^{n+1}, V)$ and $\ell^\infty((G/P)^{n+1}, V)$. These functoriality statements require the existence of a contracting homotopy on the complex $\mesw((G/P)^{n+1}, V)$, which is provided by evaluation of the first variable on any given point.
\end{proof}


\begin{proof}[Proof of Theorem~\ref{thm:funct}]
We apply the above discussion to $G=\PGL_n(\RR)$, $\Gamma=\PGL_n(\ZZ)$ and $V=\RRo$; we let $P$ be the stabiliser of a complete flag (i.e.\ a minimal parabolic) so that $G/P\cong \dr(\RR^n)$. Now $\coc$ is measurable and is a $G$-equivariant cocycle by Corollary~\ref{cor:coc}; in particular Lemma~\ref{lem:compare} holds for $\omega=\coc$. Since the restriction to any lattice preserves the semi-norm~\cite[8.6.2]{Monod}, we conclude
$$\|[q\coc]\bdd\| = \|[i\coc_\Gamma]\bdd\| = \inf \big\| i\coc + d b\big\|_{\ell^\infty},$$
where $b$ ranges over all bounded $\PGL_n(\ZZ)$-equivariant maps $b:(\dr(\RR^{n})^{n}\rightarrow\RRo$. By Lemma~\ref{lem:coboundaries}, the coboundary $db$ vanishes on a specific hereditarily spanning $(n+1)$-tuple and $\coc$ has value $\pm 1$ on those tuples by Corollary~\ref{cor:coc}. Thus $\|[q\coc]\bdd\| =1$. Now that we established this in the bounded cohomology of $\PGL_n(\RR)$, the proposition follows since the quotient map $\GL_n(\RR)\to \PGL_n(\RR)$ induces an isometric isomorphism in bounded cohomology~\cite[8.5.2]{Monod}.
\end{proof}

\section{The bounded Euler class}\label{sec:eub}

We have worked throughout with the module $\RRo$ over the group $\GL_n(\RR)$, whilst the Introduction dealt more classically with the trivial module $\RR$ over $\GL_n^+(\RR)$. Our goal in this section is to reconcile the viewpoints by deducing Theorem~\ref{thm:eub} from the following.

\begin{thm}\label{thm:coho:proj}
$\hb^q(\GL_n(\RR), \RRo)$ is one-dimensional for $q=n$ and vanishes for $q< n$.
\end{thm}

\begin{proof}[Proof that Theorem~\ref{thm:coho:proj} implies Theorem~\ref{thm:eub}]
Let $G$ be a locally compact group with an index-two closed subgroup $G^+<G$. On the one hand, we have a decomposition of the cohomology $\hhb^\bu(G^+, \RR)$ (bounded or not) as the sum of the symmetric and antisymmetric subspaces as described in the Introduction. On the other hand, there are induction isomorphisms identifying $\hhb^\bu(G^+, \RR)$ with the cohomology of $G$ with values in the module of maps $G/G^+\to \RR$, which itself is simply $\RR\oplus \RRo$ as a $G$-module (where $\sig$ is the unique non-trivial character of $G$ that is trivial on $G^+$). The two decompositions coincide, and moreover the restriction map
$$\hb^\bu(G, \RRo) \lra \hb^\bu(G^+, \RR)^{G/G^+}$$
is an \emph{isometric} isomorphism onto the subspace of antisymmetric classes, see~\cite[8.8.5]{Monod}. It follows that the corresponding restriction map in usual cohomology also preserves the norm. Specialising to our setting, Theorem~\ref{thm:coho:proj} implies that the composed map
$$\hb^n(\GL_n(\RR), \RRo) \lra \hh^n(\GL_n(\RR), \RRo) \lra \hh^n(\GL_n^+(\RR), \RR)$$
is an isometric isomorphism onto the subspace of antisymmetric classes. Thus Theorem~\ref{thm:eub} follows since $\eu$ is antisymmetric and is in the image of bounded cohomology~\cite{Gromov, Ivanov-Turaev}.
\end{proof}

To prove Theorem~\ref{thm:coho:proj}, we will appeal to~\cite{MonodMRL}; for other Lie groups, one would try to use~\cite{MonodVT}.

\begin{proof}[Proof of Theorem~\ref{thm:coho:proj}]
We introduce temporarily the following notation. Let $G$ be $\GL_n(\RR)$, let $Q<G$ be the stabiliser of the projective point corresponding to $e_1$ in $\pr(\RR^n)$ and $N\lhd Q$ be the normal subgroup isomorphic to $\RR^{n-1}\rtimes \{\pm 1\}$ given by all matrices of the form
$$\left(\begin{array}{cc}
\pm 1 & v_2 \ldots v_n\\
0 & \mbox{Id}_{n-1}\end{array}\right) \kern5mm (v_i\in\RR)$$
We identify $G/Q$ with $\pr(\RR^n)$. According to Theorem~5 in~\cite{MonodMRL}, $\hb^\bu(G, \RRo)$ vanishes in degrees~$\leq n-1$ and is realized in all degrees by the complex
$$0\lra L^\infty(\pr(\RR^n), \RRo)^G \lra  L^\infty(\pr(\RR^n)^2, \RRo)^G \lra \cdots$$
provided three conditions $(\mathrm{M}_\mathrm{I})$, $(\mathrm{M}_\mathrm{II})$ and $(\mathrm{A})$ are satisfied (the statement in \emph{loc.\ cit.} does not provide isometric isomorphisms).

\smallskip

Condition~$(\mathrm{M}_\mathrm{I})$ states that the stabiliser in $N$ of a.e.\ point in $\pr(\RR^n)^{n-2}$ has no non-zero invariant vector in $\RRo$ --- which in the case of $\RRo$ just means that this stabiliser should contain an element of negative determinant. The stabiliser in $N$ of any projective point given by a vector $x\neq e_1$ is determined by the equation $\sum_{i=2}^n x_i v_i =  (1- \pm 1) x_1$. Therefore, choosing $-1$ to ensure negative determinant, we see that a generic $(n-2)$-tuple of points in $\pr(\RR^n)$ is stabilised whenever $(v_2, \cdots, v_n)$ is in the intersection of $(n-2)$ affine hyperplanes in $\RR^{n-1}$, whose linear parts are generic. Thus there is a whole affine line of matrices with negative determinant in this stabiliser. This verifies the condition.

\smallskip

Condition~$(\mathrm{M}_\mathrm{II})$ requires that the stabiliser in $G$ of a.e.\ point in $\pr(\RR^n)^n$ has no non-zero invariant vector in $\RRo$. This is so since for any basis of $\RR^n$ the stabiliser of the corresponding projective points is conjugated to the diagonal subgroup.

\smallskip

Condition~$(\mathrm{A})$ demands that the $G$-action on $\pr(\RR^n)^n$ be amenable in Zimmer's sense; this follows from the amenability of the generic stabiliser (which we just identified as a commutative group) in view of the criterion given in~\cite[Theorem~A]{Adams-Elliott-Giordano} and of the fact that the action has locally closed orbits.
\end{proof}

\section{Relation to the simplicial cocycles of Sullivan and Smillie}
\label{sec:ss}
Let us summarize what we established so far. The space $\hb^n(\GL_n(\RR), \RRo)$ is one-dimensional and thus generated by a class $\eub$ which maps isometrically to $\eu$ in $\hh^n(\GL^+_n(\RR), \RR)$ (Section~\ref{sec:eub}). On the other hand, $\hb^n(\GL_n(\RR), \RRo)$ contains the element $[q\coc]\bdd=[q\coco]\bdd$ which has norm one (Section~\ref{sec:funct}). 

Therefore, it remains only to determine the proportionality constant between $\eub$ and $[q\coco]\bdd$; both Theorem~\ref{thm:norm} and Theorem~\ref{thm:two-values} then follow.  We shall do so by describing explicitly in Proposition~\ref{prop:AmapstoSmillie} how $\coco$ relates to the simplicial cocycles constructed by Sullivan~\cite{Sullivan76} and Smillie~\cite{Smillie_unpublished} for the Euler class of a flat $\GL^+_{n}(\RR)$-bundle over a simplicial complex. At the end of the section, we explain the relation with the Ivanov--Turaev cocycle~\cite{Ivanov-Turaev}.

\bigskip

We start by recalling the constructions of Sullivan and Smillie. Let $\xi$ be a flat $\GL^+_{n}(\RR)$-bundle over the geometric realization $|K|$ of a finite simplicial complex $K$. Let $V$ be the corresponding oriented $n$-vector bundle over $|K|$. Since the bundle $V$ is trivial if and only if there exists $n$ linearly independent sections, it is natural to start by finding one non-vanishing section. It is always possible to define such a section $s$ on the $(n-1)$-skeleton of $K$ because $\pi_{i}(\RR^{n}\backslash\{0\})$ is trivial for $i=0,\ldots,n-2$. However, this section may not be extensible to the $n$-skeleton of $K$. Thus, one defines a simplicial $n$-cochain on $K$ by assigning to every oriented $n$-simplex $k$ of $K$ the integer in $\ZZ\cong\pi_{n-1}(\mathbf{S}^{n-1})$ defined as the degree of the map
\[
\mathbf{S}^{n-1}\simeq\partial k\xrightarrow{\ \ s\ \ } \RR^{n} \backslash\{0\}\simeq \mathbf{S}^{n-1},
\]
where we chose an orientation-preserving trivialization $V|_{k}\cong |k|\times \RR^n$. Since the vector bundle $V$ is oriented, this construction is well defined; it yields a cocycle representing the Euler class in $\hh_{\text{simpl}}^{n}(K,\ZZ)$.

Sullivan observed~\cite{Sullivan76} that when the bundle $\xi$ is flat, the section $s$ can be chosen to be affine on each $(n-1)$-simplex of $K$. Thus, the map $\mathbf{S}^{n-1}\rightarrow \mathbf{S}^{n-1}$ can wrap at most once around the origin, so that the resulting cocycle $\ssul(s)$ takes values in $\{-1,0,1\}$.

Smillie later improved Sullivan's bounds as follows~\cite{Smillie_unpublished}: The locally affine section only depends on its values on the vertices $x_{1},\ldots,x_{r}$ of $K$. Choosing non-vanishing vectors $v_{i}$ in the fiber over $x_{i}$ hence defines a section $s=s(v_{1},\ldots,v_{r})$, which in the generic case will be a non-vanishing section on the $(n-1)$-skeleton. One can then form the average
\[
\ssmi(s)=2^{-r} \sum_{\sigma_{i}=\pm 1}\ssul(s(\sigma_{1} v_{1},\ldots,\sigma_{r} v_{r}))
\]
over all sign choices. This improves Sullivan's bound by a factor $2^n$ because the value at any given simplex only depends on the $n+1$ signs of the corresponding vertices, and exactly two of these signs contribute non-trivially.

The cocycles of Sullivan and Smillie also admit the following alternative description. Define a map $\sul: (\RR^{n})^{n+1}\to \RRo$ as follows: $\sul(v_0,\ldots,v_n)$ vanishes if $0$ is not contained in the interior of the convex hull of $v_0,\ldots,v_n$; in particular $\sul$ vanishes on non hereditarily spanning vectors. If $0$ does belong to the interior of the convex hull of $v_0,\ldots,v_n$, then set
$$\sul(v_0,\ldots,v_n)=(-1)^i \ori(v_0,\ldots,\widehat{v_i},\ldots,v_n),$$
where $i$ is arbitrary in $\{0,\ldots,n\}$; since $0$ belongs to the interior of the convex hull, this definition is independent of $i$. Clearly, $\sul$ is $\GL_n(\RR)$-equivariant and alternating. Observe that the evaluation of Sullivan's simplicial cocycle $\ssul(s)$ on an $n$-simplex with vertices $x_{i_0},\ldots,x_{i_n}$ can be rewritten as
$$ \ssul(s)(\langle x_{i_0},\ldots,x_{i_n}\rangle)=\sul(\psi s(x_{i_0}),\ldots,\psi s(x_{i_n})),$$
where $\psi: V|_{\langle x_{i_0},\ldots,x_{i_n}\rangle} \cong \langle x_{i_0},\ldots,x_{i_n}\rangle \times \RR^n \rightarrow \RR^n$ is any (orientation preserving) trivialization over $\langle x_{i_0},\ldots,x_{i_n}\rangle$ followed by the canonical projection.  

The fact that $\ssul(s)$ is a cocycle for $s$ generic --- well known from obstruction theory --- can easily be proved directly under the above identification:

\begin{prop} \label{prop:dsul=0}
Let $v_{0},\ldots,v_{n+1}\in\RR^{n}$ be hereditarily spanning. Then
\[
d \sul(v_{0},\ldots,v_{n+1})=0.\]
\end{prop}

\begin{proof}
We can assume that there is some $i$ with $\sul(v_{0},\ldots,\widehat{v_{i}},\ldots,v_{n+1})\neq0$. Since $\sul$ is alternating, so is $d\sul$ and we can without loss of generality suppose that $\sul(v_{1},\ldots,v_{n+1})\neq0$. Define cones $C_{i}$ in $\RR^{n}$ by
\[
C_{i}=\Bigg\{ -\sum_{\substack{j=1\\j\neq i}}^{n+1}t_{j} v_{j} : \text{ }t_{j}>0\Bigg\}. \kern10mm (1\leq i \leq n+1)
\]
Since $\sul(v_1,\ldots,v_{n+1})\neq 0$, the cones $C_i$ are open, disjoint, and their closures cover $\RR^n$. Indeed, the affine simplex with vertices $-v_{1},\ldots,-v_{n+1}$ (also) contains $0$ and $C_{i}$ is the open cone defined by the face with vertices $-v_{1}\ldots,\widehat{-v_{i}},\ldots,-v_{n+1}$.

It now remains to see where the point $v_{0}\in\RR^{n}$ belongs to. We first show that $v_0$ does not belong to the boundary of any of the $C_i$'s. Indeed, if this were the case, then $v_0$ would be a linear combination of strictly less than $n$ of the vectors $v_1,\ldots,v_{n+1}$, which would contradict the assumption that the vectors are hereditarily spanning. Thus, there exists a unique $j\in \{1,\ldots,n+1\}$ such that $v_0\in C_j$. Observe that
$$ v_0\in C_i \Longleftrightarrow \sul(v_0,v_1,\ldots,\widehat{v_i},\ldots,v_{n+1})\neq 0,$$
so that the cocycle relation simplifies to
\begin{align*}
&d \sul(v_{0},\ldots,v_{n+1}) =\\
=\ &\sul(v_{1},\ldots,v_{n+1})+(-1)^{j}\sul(v_{0} ,\ldots,\widehat{v_{j}},\ldots,v_{n+1})\\
=\ &(-1)^{j-1}\ori(v_1,\ldots,\widehat{v_{j}},\ldots,v_{n+1})+(-1)^j\ori(v_1,\ldots,\widehat{v_{j}},\ldots,v_{n+1})=0.
\end{align*}
\end{proof}

Smillie's improvement~\cite{Smillie_unpublished} on the Milnor--Sullivan bounds suggests to consider the average of $\sul$ over all possible sign changes. It is straightforward to check that the resulting map descends to the projective space and retains the other desirable properties:

\begin{lem}\label{lem:Smillie}
The map $\smi: (\pr(\RR^{n}))^{n+1}\to \RRo$ defined by
$$\smi(v_0, \ldots, v_n) = 2^{-(n+1)} \sum_{\sigma_i=\pm 1} \sul(\sigma_0 v_0, \ldots, \sigma_n v_n)$$
is $\GL_n(\RR)$-equivariant and its coboundary vanishes on hereditarily spanning $(n+2)$-tuples.\qed
\end{lem}

It is also easy to check that for $v_i=e_i$ the only non-zero summands are $\sul(-e_0, e_1, \ldots, e_n)= 1$ and $\sul(e_0, -e_1, \ldots, -e_n)=1$. Therefore, recalling that $\pcoc(e_{0},e_{1},\ldots,e_{n}) =(-1)^{n/2}$, Lemma~\ref{lem:Smillie} and Proposition~\ref{prop:unique:pcoc} imply:

\begin{cor}\label{cor:Smillie}
We have
$$\pcoc(v_0, \ldots, v_n) = (-1)^{n/2}\, 2^n\, \smi(v_0, \ldots, v_n) $$
for all $v_i\in\RR^n$.\qed
\end{cor}

\noindent
(\emph{Nota bene}: Lemma~\ref{lem:Smillie} and Corollary~\ref{cor:Smillie} give a third proof of Proposition~\ref{prop:dPE=0}.)

\medskip

At this point it is apparent that we are ready to exhibit a proportionality relation between the class $[q\pcoc]$ in ordinary (continuous) cohomology defined be the $L^\infty$-cocycle $q\pcoc = q\coc = q\coco$ and the Euler class of flat bundles:

\begin{prop}\label{prop:AmapstoSmillie}
Let $V$ be a flat oriented $n$-vector bundle over a finite simplicial complex $K$ induced by a representation $\pi_1(|K|)\rightarrow \GL^+_n(\RR)$. Then the resulting map
$$\hh^n(\GL^+_n(\RR), \RR)\longrightarrow \hh^n_\mathrm{simpl}(K, \RR)$$
sends $(-1)^{n/2}2^{-n} [q\pcoc]$ to the (real) Euler class $\eu (V)$ of the bundle $V$. 
\end{prop}

Furthermore, we will explain in the proof how this map can be realized on cochains to yield $\ssmi(s)$ for an appropriate locally affine section $s$. At the singular level, this means that for any generic affine section, we can find a classifying map $|K|\rightarrow B\GL^+_{n}(\RR)^{\delta}$ for the flat bundle $V$ over $|K|$ such that $\coco$ maps to  $(-1)^{n/2}2^n \ssmi$ by pull-back.

\begin{cor}\label{cor:AmapstoSmillie}
The cocycle $q\pcoc$ represents $(-1)^{n/2}2^{n}$ times the bounded Euler class $\eub$ in $\hb^n(\GL_n(\RR), \RRo)$. Therefore, $[q\pcoc] = (-1)^{n/2}2^{n}\eu$ in $\hh^n(\GL_n^+(\RR), \RR)$
\end{cor}

\begin{proof}
Since $\hh^n(\GL_n(\RR),\RRo)$ is one-dimensional, it suffices in view of Proposition~\ref{prop:AmapstoSmillie} to find one flat bundle over some $n$-dimensional finite simplicial complex with non-trivial Euler class. For this, take a product of $n/2$ copies of such a $2$-dimensional flat bundle over a surface of genus $g$. Such $2$-dimensional flat bundles were exhibited by Milnor in~\cite{Milnor58}.
\end{proof}

\begin{proof}[Proof of Theorem~\ref{thm:two-values}]
Since $q\pcoc =q \coc = q \coco$, Theorem~\ref{thm:two-values} follows from Theorem~\ref{thm:dE=0} and Corollary~\ref{cor:AmapstoSmillie}.
\end{proof}

\begin{proof}[Proof of Theorem~\ref{thm:norm}]
We apply successively Theorem~\ref{thm:eub}, Corollary~\ref{cor:AmapstoSmillie} and Theorem~\ref{thm:funct}:
$$\|\eu\| =\|\eub\| = 2^{-n} \|[q\pcoc]\|\bdd = 2^{-n} \|[q\coc]\|\bdd = 2^{-n}.$$ 
\end{proof}

\begin{proof}[Proof of Proposition~\ref{prop:AmapstoSmillie}]
It was convenient to use a.e.\ cocycles since it allows to define the class $[q\pcoc]=[q\coco]$ with the much simpler function $\pcoc$. However, in order to implement explicit cochains maps, we shall need a true cocycle (this reflects the fact that the map in the statement factors through $\GL_n^+(\RR)^\delta$). Therefore, we realize $\hh^\bu(\GL_n^+(\RR), \RR)$ using the resolution $\bor(\GL_n^+(\RR)^{\bu+1})$ by Borel maps, see~\cite{Wigner73}. The cocycles $\coc$ and $\coco$ thus represent classes $[\coc], [\coco]$ which coincide with $[q \pcoc]$ in $\hh^n(\GL_n^+(\RR), \RR)$ (this follows e.g. since the inclusion of continuous cochains into a.e.\ cochains factors through $\bor$ and induces isomorphisms; as it turns out, we will evaluate $\coco$ at generic points only anyway).

Next, we describe on the cochain level how a representation $\ro: \pi_1(|K|)\to \GL^+_n(\RR)$ induces a map $\ro^*:\hh^\bu(\GL^+_n(\RR), \RR)\to \hh^\bu_\mathrm{simpl}(K, \RR)$; this amounts to an explicit implementation of the classifying map. Given a vertex $x$ of $K$, let $U_x$ be a neighbourhood of the closure of the star at $x$, small enough so that $U_x$ is contractible. Recall that the star at $x$ is the union of all the open simplices having $x$ as a vertex, so that $U_x$ contains all the closures of these simplices. Let
$$\varphi_x:V|_{U_x} \longrightarrow U_x \times \RR^n$$
be any trivialization of the flat bundle $V$ over $U_x$ and, for $x,y\in K^0$, denote by $g_{xy}:U_x\cap U_y\rightarrow \GL_n^+(\RR)$ the corresponding transition functions given by 
$$\varphi_x \varphi_y^{-1} (z,v)= (z,g_{xy}(z)v),$$
for $z\in U_x\cap U_y$ and $v\in \RR^n$. Then $\ro^*$ is induced at the cochain level by the map
$$\ro^*_\varphi: \bor(\GL_n^+(\RR)^{\bu+1})^{\GL_n^+(\RR)} \lra C^\bu_\mathrm{simpl}(K)$$
that sends a $\GL_n^+(\RR)$-invariant cochain $D$ to the simplicial cochain whose value on a simplex with vertices $x_0,\ldots,x_q$ is
$$\ro^*_\varphi(D)(\langle x_0,\ldots,x_q \rangle) = D(g_{i0},\ldots,g_{ii},\ldots,g_{iq}),$$
where $g_{ij}\in \GL_n^+(\RR)$ is the value of the transition function $g_{x_ix_j}$ on the connected component of $U_{x_i}\cap U_{x_j}$ containing $\langle x_0,\ldots,x_q \rangle$. In view of the cocycle relations of the transition functions and the fact that $D$ is $\GL_n^+(\RR)$-invariant, the definition does not depend on $i$. 

Returning to Smillie's cocycle, choose $s(x)\in V|_{\{x\}}$ for every vertex $x$ so that the resulting locally affine section is nowhere vanishing on the $(n-1)$-skeleton. Pick $0\neq v\in \RR^n$ and choose trivializations 
$$\varphi_x:V|_{U_x} \longrightarrow U_x \times \RR^n$$
such that 
$$\varphi_x(s(x))=(x,v).$$
Such trivializations are obtainable by composing, over every $U_x$, any given trivialization with an appropriate transformation of $\GL_n^+(\RR)$. Smillie's cocycle is given by 
$$\ssmi(s)(\langle x_0,\ldots,x_q \rangle)=\smi(\psi s(x_0),\ldots,\psi s(x_n)), \eqno{(*)}$$
where $\psi: V|_{ \langle x_0,\ldots,x_q \rangle} \rightarrow \langle x_0,\ldots,x_q \rangle \times \RR^n \rightarrow \RR^n$ is given by any trivialization of $V|_{\langle x_0,\ldots,x_q \rangle}$, in particular by the restriction of $\varphi_{x_0}$ to $\langle x_0,\ldots,x_q \rangle$, so that~$(*)$ rewrites as
$$\smi (v, g_{01}v,\ldots,g_{0n}v)$$
which in turn is by definition equal to 
$$\ro^*_\varphi\Big( (-1)^{n/2}2^{-n} \coco \Big)(\langle x_0,\ldots,x_q \rangle).$$ 
This finishes the proof of the proposition.
\end{proof}

Finally, we comment on the relation with the cocycle constructed by Ivanov--Turaev in~\cite{Ivanov-Turaev}. Expressed in the homogeneous bar resolution, the Ivanov--Turaev cocycle becomes the following map:
$$\itu(g_0, \ldots, g_n) = \int_{B^{n+1}} \sul(g_0 v_0, \ldots, g_n v_n)\, d v_0\ldots d v_n, \kern10mm(g_i\in\GL_n(\RR))$$
where $B$ is the unit ball in $\RR^n$ with normalised measure. In fact they considered $\GL_n^+(\RR)$ but this is equivalent since these classes are antisymmetric.

\begin{prop}\label{prop:IT}
The class $[\itu]\bdd\in\hb^n(\GL_n(\RR), \RRo)$ defined by $\itu$ coincides with $(-1)^{n/2}\, 2^{-n}[q\pcoc]\bdd$.
\end{prop}

\begin{proof}
The above integral representation of~$\itu$ can be re-written as
$$\itu(g_0, \ldots, g_n) = \int_{\pr(\RR^n)^{n+1}} \smi(g_0 v_0, \ldots, g_n v_n)\, d v_0\ldots d v_n,$$
where now $d v_i$ is the normalised measure on the projective space. By Corollary~\ref{cor:Smillie},
$$(-1)^{n/2}\, 2^{n}\, \itu(g_0, \ldots, g_n) = \int_{\pr(\RR^n)^{n+1}} \pcoc(g_0 v_0, \ldots, g_n v_n)\, d v_0\ldots d v_n.$$
The latter integral is the Poisson transform of $\pcoc$, or more precisely of the pull-back of $\pcoc$ to $\dr(\RR^n)^{n+1}$, since $\dr(\RR^n)$ is a Poisson boundary for $\GL_n(\RR)$. It is a general property of Poisson transforms that this class coincides with $[q\coc]\bdd$ in bounded cohomology, see~\cite[7.5.8]{Monod}.
\end{proof}

Ivanov--Turaev proved $[\itu] = \eu$ in~\cite{Ivanov-Turaev}. Their proof is based on an analogue of Proposition~\ref{prop:AmapstoSmillie} for the cocycle $\itu$, see Theorem~2 (finite case) in~\cite{Ivanov-Turaev}. In light of Proposition~\ref{prop:IT}, the two approaches are essentially equivalent. Therefore, we could have avoided the explicit proof of Proposition~\ref{prop:AmapstoSmillie} and Corollary~\ref{cor:AmapstoSmillie} by first establishing Proposition~\ref{prop:IT} and then quoting~\cite{Ivanov-Turaev}. Conversely, Corollary~\ref{cor:AmapstoSmillie} yields an alternative proof that the Ivanov--Turaev cocycle represents the Euler class.

\bibliographystyle{../BIB/amsalpha}
\bibliography{../BIB/ma_bib}

\def\cprime{$'$}
\providecommand{\bysame}{\leavevmode\hbox to3em{\hrulefill}\thinspace}
\begin{thebibliography}{AEG94}

\bibitem[AEG94]{Adams-Elliott-Giordano}
Scot Adams, George~A. Elliott, and Thierry Giordano, \emph{Amenable actions of
  groups}, Trans. Amer. Math. Soc. \textbf{344} (1994), no.~2, 803--822.

\bibitem[BG08]{Bucher-Gelander08}
Michelle Bucher and Tsachik Gelander, \emph{Milnor-{W}ood inequalities for
  manifolds locally isometric to a product of hyperbolic planes}, C. R. Math.
  Acad. Sci. Paris \textbf{346} (2008), no.~11-12, 661--666.

\bibitem[BG09]{Bucher-GelanderARX}
Michelle Bucher and Tsachik Gelander, \emph{The generalized {C}hern conjecture
  for manifolds that are locally a product of surfaces}, 2009, Preprint,
  arXiv:0902.1215.

\bibitem[BI04]{Burger-IozziPSUpq}
Marc Burger and Alessandra Iozzi, \emph{Bounded {K}\"ahler class rigidity of
  actions on {H}ermitian symmetric spaces}, Ann. Sci. \'Ecole Norm. Sup. (4)
  \textbf{37} (2004), no.~1, 77--103.

\bibitem[BM99]{Burger-Monod1}
Marc Burger and Nicolas Monod, \emph{Bounded cohomology of lattices in higher
  rank {L}ie groups}, J. Eur. Math. Soc. (JEMS) \textbf{1} (1999), no.~2,
  199--235.

\bibitem[BM02]{Burger-Monod3}
Marc Burger and Nicolas Monod, \emph{Continuous bounded cohomology and
  applications to rigidity theory}, Geom. Funct. Anal. \textbf{12} (2002),
  no.~2, 219--280.

\bibitem[Buc04]{BucherKarlsson}
Michelle Bucher, \emph{Characteristic classes and bounded cohomology}, Ph.D.
  thesis, ETHZ Diss. Nr.~15636, 2004.

\bibitem[Buc07]{Bucher07}
Michelle Bucher, \emph{Finiteness properties of characteristic classes of flat
  bundles}, Enseign. Math. (2) \textbf{53} (2007), no.~1-2, 33--66.

\bibitem[Buc08]{Bucher08}
Michelle Bucher, \emph{The simplicial volume of closed manifolds covered by
  {$\Bbb H^2\times\Bbb H^2$}}, J. Topol. \textbf{1} (2008), no.~3, 584--602.

\bibitem[C{\O}03]{Clerc-Orsted03}
Jean-Louis Clerc and Bent {\O}rsted, \emph{The {G}romov norm of the {K}aehler
  class and the {M}aslov index}, Asian J. Math. \textbf{7} (2003), no.~2,
  269--295.

\bibitem[DT87]{Domic-Toledo}
Antun Domic and Domingo Toledo, \emph{The {G}romov norm of the {K}aehler class
  of symmetric domains}, Math. Ann. \textbf{276} (1987), no.~3, 425--432.

\bibitem[Dup79]{Dupont}
Johan~L. Dupont, \emph{Bounds for characteristic numbers of flat bundles},
  Algebraic topology, Aarhus 1978 (Proc. Sympos., Univ. Aarhus, Aarhus, 1978),
  Springer, Berlin, 1979, pp.~109--119.

\bibitem[Ghy87]{Ghys84}
{\'E}tienne Ghys, \emph{Groupes d'hom\'eomorphismes du cercle et cohomologie
  born\'ee}, The Lefschetz centennial conference, Part III (Mexico City, 1984),
  Amer. Math. Soc., Providence, RI, 1987, pp.~81--106.

\bibitem[Gro82]{Gromov}
Micha{\"\i}l Gromov, \emph{Volume and bounded cohomology}, Inst. Hautes
  \'Etudes Sci. Publ. Math. (1982), no.~56, 5--99 (1983).

\bibitem[Hir58]{Hirzebruch58}
Friedrich Hirzebruch, \emph{Automorphe {F}ormen und der {S}atz von
  {R}iemann-{R}och}, Symposium internacional de topolog\'\i a algebraica
  {I}nternational symposium on algebraic topology, Universidad Nacional
  Aut\'onoma de M\'exico and UNESCO, Mexico City, 1958, pp.~129--144.

\bibitem[IT82]{Ivanov-Turaev}
Nikolai~V. Ivanov and Vladimir~G. Turaev, \emph{The canonical cocycle for the
  {E}uler class of a flat vector bundle}, Dokl. Akad. Nauk SSSR \textbf{265}
  (1982), no.~3, 521--524.

\bibitem[Mil58]{Milnor58}
John Milnor, \emph{On the existence of a connection with curvature zero},
  Comment. Math. Helv. \textbf{32} (1958), 215--223.

\bibitem[Mon01]{Monod}
Nicolas Monod, \emph{{Continuous bounded cohomology of locally compact
  groups}}, {Lecture Notes in Mathematics 1758, Springer, Berlin}, 2001.

\bibitem[Mon06]{MonodICM}
Nicolas Monod, \emph{{An invitation to bounded cohomology}}, {Proceedings of
  the international congress of mathematicians (ICM), Madrid, Spain, August
  22--30, 2006. Volume II: Invited lectures. Z\"urich: European Mathematical
  Society, 1183--1211}, 2006.

\bibitem[Mon07]{MonodMRL}
Nicolas Monod, \emph{Vanishing up to the rank in bounded cohomology}, Math.
  Res. Lett. \textbf{14} (2007), no.~4, 681--687.

\bibitem[Mon10]{MonodVT}
Nicolas Monod, \emph{On the bounded cohomology of semi-simple groups,
  {S}-arithmetic groups and products}, J. Reine Angew. Math. \textbf{640}
  (2010), 167--202.

\bibitem[Sim96]{Simonnet}
Michel Simonnet, \emph{Measures and probabilities}, Universitext,
  Springer-Verlag, New York, 1996, With a foreword by Charles-Michel Marle.

\bibitem[Smi]{Smillie_unpublished}
John Smillie, \emph{The {E}uler characteristic of flat bundles}, unpublished
  manuscript.

\bibitem[Sul76]{Sullivan76}
Dennis Sullivan, \emph{A generalization of {M}ilnor's inequality concerning
  affine foliations and affine manifolds}, Comment. Math. Helv. \textbf{51}
  (1976), no.~2, 183--189.

\bibitem[Thu78]{Thurston_unpublished}
William~P. Thurston, \emph{Geometry and topology of 3-manifolds}, Princeton
  notes, 1978.

\bibitem[Wig73]{Wigner73}
David Wigner, \emph{Algebraic cohomology of topological groups}, Trans. Amer.
  Math. Soc. \textbf{178} (1973), 83--93.

\bibitem[Woo71]{Wood71}
John~W. Wood, \emph{Bundles with totally disconnected structure group},
  Comment. Math. Helv. \textbf{46} (1971), 257--273.

\end{thebibliography}
\end{document}